\documentclass[12pt,reqno, oneside]{amsart}
\usepackage{amsmath,amssymb,latexsym,esint,cite,mathrsfs,slashed, dsfont}
\usepackage{verbatim,cite,relsize,color,soul}
\usepackage[latin1]{inputenc}
\usepackage{microtype}
\usepackage{color,enumitem,graphicx, xcolor}
\usepackage{mathtools}
\mathtoolsset{showonlyrefs}
\usepackage[colorlinks=true,urlcolor=blue, citecolor=red,linkcolor=blue,
linktocpage,pdfpagelabels, bookmarksnumbered,bookmarksopen]{hyperref}
\usepackage[hyperpageref]{backref}
\usepackage[english]{babel}
\usepackage{geometry}

\numberwithin{equation}{section}
\newtheorem{theorem}{Theorem}[section]
\newtheorem{lemma}[theorem]{Lemma}

\newtheorem{prop}[theorem]{Proposition}
\theoremstyle{definition}

\newtheorem{remark}[theorem]{Remark}
\newtheorem*{ackno}{Acknowledgements}

\newcommand{\R}{\mathbb R}

\def\({\left(}
\def\){\right)}
\def\<{\left\langle}
\def\>{\right\rangle}
\DeclareMathOperator*{\loc}{loc}
\DeclareMathOperator*{\GN}{GN}

\DeclareMathOperator*{\rea}{Re}
\DeclareMathOperator*{\di}{div}

\title[LWP and blow-up NLS point interaction]
{Local well-posedness and blow-up  in the energy space for the 2D NLS with point interaction}
\author[L. Forcella, V. Georgiev
]{Luigi Forcella and Vladimir Georgiev}

\address[L. Forcella]{Dipartimento di Matematica, Universit\`a Degli Studi di Pisa, Largo Bruno Pontecorvo, 5, 56127, Pisa, Italy}
\email{luigi.forcella@unipi.it}

\address[V. Georgiev]{Dipartimento di Matematica, Universit\`a Degli Studi di Pisa, Largo Bruno Pontecorvo, 5, 56127, Pisa, Italy, and Faculty of Science and Engineering, Waseda University, 3-4-1, Okubo, Shinjuku-ku, Tokyo 169-8555, Japan, and Institute of Mathematics and Informatics,  Bulgarian Academy of Sciences,  Acad. Georgi Bonchev Str., Block 8, 1113 Sofia, Bulgaria}

\email{vladimir.simeonov.gueorguiev@unipi.it}



\begin{document}
	
\begin{abstract}
We consider the two-dimensional nonlinear Schr\"odinger equation with point interaction and we establish a local well-posedness theory, including blow-up alternative and continuous dependence on the initial data in the energy space. We provide a proof by employing a Kato's method along with Hardy inequalities with logarithmic correction. Moreover, we establish finite time blow-up for solutions with positive energy and infinite variance.
\end{abstract}

\maketitle
	
\section{Introduction}

In this paper, we consider the Cauchy problem
\begin{equation} \label{NLS}
		\begin{cases}
			i\partial_t u = \Delta_{\alpha}u\pm |u|^{p-1}u,\\ 
			u(0,x)= u_0(x),
		\end{cases}
\end{equation}
where $u=u(t,x)$, $u:\mathbb{R}\times \mathbb{R}^2\mapsto \mathbb{C}$, $p>1$ and the operators $\Delta_\alpha$ are a family of self-adjoint extensions of the free Laplace operator $-\Delta$ on the domain $C^{\infty}(\mathbb R^2\setminus\{0\})$.  The parameter  $\alpha$ varies in $\mathbb R\cup \{+\infty\}$, and characterizes  all nontrivial self-adjoint extensions on $L^2(\mathbb R^2)$ of $-\Delta\lvert_{C^{\infty}(\mathbb R^2\setminus\{0\})}$.  When $\alpha=+\infty$ we mean $\Delta_{\infty}=-\Delta$.  It can be thought  like a perturbation of the free Laplacian with a point defect (or point interaction). It is also sometimes incorrectly referred to as the Schr\"odinger operator with delta interaction, as it can be viewed as a two-dimensional extension  of the one-dimensional  operator $-\frac{d^2}{dx^2}+\alpha\delta$, where $\delta$ represents the Dirac delta distribution. Generally, the operators $\Delta_\alpha$ are well defined only in $\mathbb R^d$ with $d=1,2,3$,  and we refer to the monograph \cite{AGKH05} for a comprehensive source on that topic. In this paper, we focus on the two-dimensional case for reasons that will become clear later (see Remark \ref{rem-char} and the proof of Theorem \ref{thm-lwp}  in Section \ref{sec:LWP}). For a rigorous construction of the two-dimensional case refer to \cite{AGKH87}. Additionally, the exhaustive introductions  in \cite{CFN2021, FGI22} provide further insights.\vspace{3mm}

On the physical side, point interactions can represent idealized models of particles interacting through very short-range forces.  This is particularly useful in studying systems like quantum dots or impurities in a crystal lattice. One can see in \cite{BZ11} the link between contact-type regimes and two-impurity Kondo system, for example.  We refer to  \cite[Chapter I.5]{AGKH05} or \cite{AH81} for the mathematical framework and physical implications of point interactions in two dimensions. \vspace{3mm}

As stated before, the point-like perturbation of the standard Laplace operator can be modelled by a  one-parameter family  $\Delta_\alpha$, $\alpha\in(-\infty,\infty]$ of self-adjoint extensions of the free Laplace operator $-\Delta$ on the domain $C^{\infty}(\mathbb R^2\setminus\{0\})$. The spectrum of this operator has an absolutely continuous part $[0,\infty)$ and only one real negative eigenvalue
$e_\alpha = -4e^{-2(2\pi\alpha + \gamma)}.$
The natural energy space $H^1_\alpha$ associated with  \eqref{NLS} is the domain of $(\lambda+\Delta_\alpha)^{1/2},$ where $\lambda > |e_{\alpha}|.$  A typical element in the space  $u\in H^1_\alpha$ has the structure $u=\phi^\lambda+q\mathbb G^\lambda$, where $\phi^\lambda\in H^1(\mathbb{R}^2)$, $q\in\mathbb C$, and $\mathbb G^\lambda\notin H^1(\mathbb R^2)$. See Section \ref{Sec-tool} for rigorous definitions.
 The Cauchy datum $u_0$ in \eqref{NLS} is taken this space. Note by the form of a element $u\in H^1_\alpha(\mathbb R^2)$, this space strictly contains the usual Sobolev space $H^1(\mathbb R^2)$ arising from the unperturbed Laplacian.
Our goal is to establish a local existence  theory for \eqref{NLS} with $u_0\in H^1_{\alpha}(\mathbb R^2)$, and our first main result is the following.

\begin{theorem}\label{thm-lwp}
Let $\alpha\in \mathbb R$ and $p>1$. The Cauchy problem \eqref{NLS} is locally well posed in the classical sense in the energy space. More precisely:\vspace{2mm}

\noindent\textit{(i)} for any $u_0\in H^1_{\alpha}(\mathbb{R}^2)$ there exists a unique solution $u(t)\in C((-T_{\min}, T_{\max}))$ to \eqref{NLS} where   $(-T_{\min}, T_{\max})$ is the maximal time interval of existence, with $T_{\min}, T_{\max}>0$;\medskip
    
 \noindent    \textit{(ii)} the mass and the energy are conserved quantities, i.e.,  for any $ t\in (-T_{\min},T_{\max})$
 \begin{equation}\label{eq:mass}
M(u(t)):=\|u(t)\|_{L^2(\mathbb{R}^2)}^2=M(u_0)
\end{equation}
and 
\begin{equation}\label{eq:energy}
E(u(t)):=\frac12F(u(t))-\frac{1}{p+1}\|u(t)\|_{L^{p+1}(\mathbb{R}^2)}^{p+1}=E(u_0) 
\end{equation}
where, for $u$ of the form $u=\phi^\lambda+q\mathbb G^\lambda$ and $\Gamma^\lambda_\alpha$ a constant depending on $\alpha$ and $\lambda$,
\begin{equation}\label{quad-form}
F(u(t))=F_\alpha(u(t)):=\|\nabla\phi^\lambda\|_{L^2(\mathbb{R}^2)}^2+\lambda(\|\phi^\lambda\|_{L^2(\mathbb{R}^2)}^2-\|u\|_{L^2(\mathbb{R}^2)}^2)+\Gamma^\lambda_\alpha|q|^2, 
\end{equation}
respectively; \medskip

\noindent   \textit{(iii)} the solution also belongs to the Strichartz spaces $L_{\loc}^{q}((-T_{\min}, T_{\max}); H_{\alpha}^{1,\rho}(\mathbb R^2))$ where $(q,\rho)$ is any Strichartz admissible pair, i.e. $\frac1q=\frac12-\frac1\rho$, $\rho\geq 2$;\medskip
     
 \noindent    \textit{(iv)} we have the blow-up alternative, namely either the solution is global or $T_{\max}>0$ is finite and 
\[
\lim_{t\nearrow T_{\max}}\|u(t)\|_{H^1_{\alpha}(\mathbb{R}^2)}=+\infty.
\] 
A similar statement holds for $T_{\min}$;\medskip
    
\noindent    \textit{(v)} the solution map is continuous with respect to the initial data.
\end{theorem}

The strategy to prove Theorem \ref{thm-lwp} is to implement  Kato's approach based on the validity of certain  a-priori estimates on the nonlinearity, see conditions \eqref{kato1} and \eqref{kato2} in Section \ref{sec:LWP}. To achieve this, our main ideas are the followings:\smallskip

\noindent\textup{(i)} first, we reduce the proof of the crucial estimates to a classical Sobolev space framework, by means of a characterization of singular Sobolev spaces for certain index; \smallskip

\noindent\textup{(ii)} secondly, we remove the singularity of $\log(|x|)$-type from the singular part of the solution (see Section \ref{Sec-tool} for the structure of an energy solution to the Cauchy problem  \eqref{NLS}), by making it appear  as a factor of the regular part of the solution. Of course, this does not  completely remove the singularity, but by using a generalized Hardy-type inequality, we can properly control the new singular term.  Consequently, we can prove the validity of Kato's hypothesis and establish the well-posedness result of Theorem \ref{thm-lwp}. \medskip 

\noindent It is worth noticing that Strichartz estimates for the linear operator $e^{it\Delta_{\alpha}}$ also play a fundamental role when proving existence results. For the operator $\Delta_{\alpha}$ in the two-dimensional case, such estimates have been established recently in \cite{CMY19, CMY19b, Ya2021}.\\

It is noteworthy that a version of  Theorem \ref{thm-lwp} was previously proven in a paper by Fukaya, Ikeda, and the second author (see \cite{FGI22}). The strategy used to establish local well-posedness in \cite{FGI22} is based on the general abstract theory developed in Okazawa, Suzuki, and Yokota (see \cite{OSY12}), which extends the results of Cazenave \cite{Cazenave}. The approach in \cite{OSY12} is quite involved, and to prove the existence, regularity, and stability results as in Theorem \ref{thm-lwp}, several conditions need to be satisfied. Additionally, the theory in \cite{OSY12} was specifically motivated by the infeasibility of Kato's method for the NLS equation with an inverse square potential. Therefore, we believe our approach is more suitable for addressing \eqref{NLS}, providing a shorter and more elegant proof. Furthermore, the method of \cite{OSY12} heavily relies on a regularization procedure, which we completely avoid here by using Kato's approach. \\
\noindent We also point-out that neither in \cite{FGI22} nor in the original paper \cite{OSY12} the blow-up alternative (i.e.,  Theorem \ref{thm-lwp} point \textit{(iv)}) was mentioned. Indeed, the blow-up alternative follows directly from the fact that the time interval of existence depends on the norm of the initial value. Indeed, this is clearly evident when applying a fixed point argument in the spirit of Kato (see \cite{Cazenave}).\smallskip

In conclusion, regarding the local well-posedness problem, we note that existence results were provided in the domain space (see \eqref{def-H2-alpha} below) in \cite{CFN2021}. With our approach, we are able to establish a local well-posedness theory in the more natural energy space.\\

The second result  in our paper concerns long time behaviour of solutions to \eqref{NLS} in the focusing regime, namely when the minus sign is considered in the nonlinearity of \eqref{NLS}. Specifically, we investigate the existence of infinite-variance, finite-time blowing-up solutions for positive energies. Before stating it, we introduce a few quantities  which will be explained further in Section \ref{Sec-tool}: for $u\in H^1_{\alpha}(\mathbb R^2)$ and $\omega\in\mathbb R$, we denote by $S_\omega(u)=E(u)+\omega M(u)$ the action functional, by $v_\omega$ a ground state, and by 
\[
P(u)=F(u)-\frac{p-1}{p+1}\|u\|_{L^{p+1}(\mathbb{R}^2)}^{p+1}+\frac{|q|^2}{4\pi}\quad \hbox{ with }\quad  u=\phi^\lambda+q\mathbb G^\lambda
\]
the Pohozaev functional. Then, our achievement reads as follows. 

\begin{theorem}\label{thm-blowup}
Let $\alpha\in \mathbb R$, $3<p\leq 5$ and consider the focusing case, i.e., the nonlinearity is $-|u|^{p-1}u$. Let $u_0\in H^1_{\alpha}(\mathbb{R}^2)$ an initial datum such that   $S_{\omega}(u_0)<S_{\omega}(v_{\omega})$, $E(u_0)\geq0$,  and $P(u_0)<0$. Suppose that $u_0(x)=u_0(|x|)$. Then $T_{\max}<\infty$. A similar statement holds for $T_{\min}.$ 
\end{theorem}

The results mentioned above concerning the unperturbed Laplace operator and negative energies traces back to the work of Ogawa and Tsutsumi (see \cite{OT91}). Their research extended the findings of Glassey (see \cite{Glassey}), who focused on finite-variance solutions. More recently, Holmer and Roudenko (see \cite{HR2007}) expanded the results of \cite{OT91} to include positive energies.\\

For the point interaction equation \eqref{NLS}, we are only aware of a recent result by Finco and Noja \cite{FN2023}. They studied solutions with finite variance, specifically those solutions to \eqref{NLS} that decay in space fast enough to satisfy the condition $u\in\Sigma:=L^2(|x|^2dx)$. For solutions with initial data in $\Sigma$, the authors of \cite{FN2023} confirmed the validity of a virial identity in the context of point interaction, similar to the one for the NLS equation with an unperturbed Laplace operator. Specifically, they demonstrated the virial identity $\displaystyle\frac{d^2}{dt^2}\int_{\mathbb R^2}|x|^2|u(t)|^2dx=8P(u(t))$. The negativity of $P(u(t))$ directly implies a blow-up result through a convexity argument, thereby proving Glassey's result in this scenario.\smallskip

In this paper, we explore a potential analogy to the results by Ogawa and Tsutsumi \cite{OT91} (and \cite{HR2007}) for \eqref{NLS} with positive energies. Theorem \ref{thm-blowup} shows that, even in the case of point interaction, the same range of nonlinearities as in \cite{OT91, HR2007} applies. Our proof also relies on virial estimates, although it is more complex  compared to the proof for solutions in $\Sigma$ and the case involving the free Laplacian. We refer to Section \ref{Sec-blowup} for more details.

\section{Preliminary tools}\label{Sec-tool}
In this section, we rigorously introduce the mathematical framework we are going to work with, we give useful properties of the singular operator appearing in \eqref{NLS}, and we show the crucial estimates used to prove our main results. Since now on, as we specialize to the two-dimensional case, we omit the space $\mathbb R^2$ when no confusion may arise.

\subsection{Properties of the  singular Laplacian operator} 
We begin with a rigorous definition of the singular Laplace operator $\Delta_\alpha$, and we introduce the functional setting used in the local well-posedness theory we develop in section \ref{sec:LWP}. \\

Let us recall the following: given the domain 
\begin{equation}\label{def-H2-alpha}
\mathcal D(\Delta_{\alpha})=\left\{u\in L^2 \hbox{ s.t. } u=\phi^\lambda+q\mathbb{G}^\lambda,  \hbox{ with } \phi^\lambda\in H^2 \hbox{ and } q=\frac{\phi^\lambda(0)}{\Gamma^\lambda_\alpha}\right\},
\end{equation}
where  
\begin{equation}\label{def-G}
\mathbb{G}^\lambda=(-\Delta+\lambda)^{-1}\delta,
\end{equation}
and 
\begin{equation}\label{def-Gamma}
\Gamma^\lambda_\alpha=\alpha+\frac{\gamma}{2\pi}+\frac{\log(\sqrt\lambda/2)}{2\pi},
\end{equation}
 $\gamma$ being the Euler-Mascheroni constant, a classical definition of the action of $\Delta_{\alpha}$ as in \cite{AGKH87} is
\begin{equation}\label{def-oper}
(\lambda + \Delta_{\alpha}) u=(\lambda-\Delta) \phi^\lambda
\end{equation}
with \eqref{def-oper} being equivalent to 
the action of $\Delta_{\alpha}$  defined as
\begin{equation}\label{def-equiv}
\Delta_{\alpha} u=-\Delta \phi^\lambda-\lambda q \mathbb{G}^\lambda.
\end{equation}
The space described in \eqref{def-H2-alpha} is denoted by $H^2_\alpha$. \smallskip

\noindent The operator  $\lambda + \Delta_{\alpha}$ with $\lambda > |e_\alpha|$ is self-adjoint and positive one. The domain of its square root can be characterized (see \cite{MO17})
\[
\mathcal D((\lambda+\Delta_\alpha)^{1/2})=\left\{u\in L^2 \hbox{ s.t. } \exists\, q\in\mathbb C \hbox{ and } \phi^\lambda\in H^1\, : \, u=\phi^\lambda+q\mathbb{G}^\lambda\right\},
\]
and this becomes a domain of the corresponding quadratic  form  
\[
F(u)=\|\nabla\phi^\lambda\|_{L^2}^2+\lambda(\|\phi^\lambda\|_{L^2}^2-\|u\|_{L^2}^2)+\Gamma^\lambda_\alpha|q|^2
\]
already appeared in \eqref{quad-form}.
Therefore,  $\mathcal D(F_\alpha)$ is the domain of the square root of the operator $\Delta_{\alpha}+\lambda$:
\begin{equation}\label{def-sing-space}
\mathcal D(F_\alpha)=\mathcal D((\Delta_{\alpha}+\lambda)^{1/2}):= H^1_\alpha,  \hbox{\quad for \quad } \lambda>-e_\alpha =|e_\alpha|,
\end{equation}
with  $e_\alpha$  being the only simple negative eigenvalue of $\Delta_{\alpha}$.
Note that   
\[
\|u\|_{H^1_\alpha}\simeq\|\phi^\lambda\|_{H^1}+|q|
\]
where $H^1$ is the classical Sobolev space $H^1=(-\Delta+1)^{-1}L^2$.\\
In this subsection we will resume some properties of the Green function $\mathbb{G}^\lambda$, the fundamental solution of $(\lambda - \Delta)$ for any
\begin{equation}\label{eq.ome1}
\lambda \in \mathbb{R} \setminus \sigma(\Delta) = \mathbb{R} \setminus (-\infty,0].
\end{equation}
First of all, it can be represented, in the 2D case, as
\begin{equation} \label{eq.defGgeneral}
\mathbb{G}^\lambda(x) = (2\pi)^{-1}  K_{0}
(\sqrt{\lambda} |x|).
\end{equation}
Here  $K_\nu$ is the second type modified Bessel function of order $\nu \geq 0$, also called Macdonald function.
The asymptotic expansions at zero and infinity are also well known. Indeed, we have
\begin{equation}\label{eq.asymptotic zero}
\mathbb{G}^\lambda(x)=-\frac{1}{2\pi}\log(|x|)-\Gamma^\lambda-\alpha+o(|x|)\ \quad\hbox{ for }\quad x\to 0,
\end{equation}
while at infinity we have the exponential decay (see \cite[relation (20), section 7.23]{W95})
\begin{equation}\label{eq.asymptotic infinity}
\mathbb{G}^1(x)=e^{-|x|}\left(\sqrt{\frac{\pi}{2|x|}} +O\left(|x|^{-3/2}\right)   \right), \quad\hbox{ for }\quad x\to \infty.
\end{equation}
We also recall that $\mathbb G^\lambda\in L^{\rho}$ for any $\rho\geq 1$.\\

For our dynamical result, two  proprieties are essential: the unique representation of an element in the domain $\mathcal D(\Delta_{\alpha})$ for a given $\lambda$, and the independence of the domain and the action with respect to the choice of $\lambda$. These are, respectively, the content of the following two lemmas, whose proofs can be found in \cite{GR24}.

\begin{lemma} For a fixed $\lambda$, suppose that for an element $u\in \mathcal D(\Delta_{\alpha})$ we have two representations
\[
\begin{aligned}
u&=\phi_1^\lambda+q_1\mathbb G^{\lambda},\\
u&=\phi_2^\lambda+q_2\mathbb G^{\lambda},
\end{aligned}
\]
where $q_j=\frac{\phi_j(0)}{\Gamma^{\lambda}_\alpha}$, $j=1,2$. Then
\[
\phi_1^\lambda=\phi_2^\lambda \hbox{\, and \,} q_1=q_2.
\]
\end{lemma}
\begin{lemma} For $\lambda>|e_\alpha|$, the domain of the operator $(\lambda+\Delta_{\alpha})$ and its action are independent of the choice of $\lambda$.
\end{lemma}

\subsection{Sobolev and Logarithmic Hardy inequalities} 

A natural extension of the $H^1_{\alpha}$ as defined in \eqref{def-sing-space} for other index of summability is 
\begin{equation}\label{def-sing-space-gen}
H^{1,r}_\alpha =\left\{u \, \hbox{ s.t. } \, \exists\, q\in\mathbb C \hbox{ and } \phi^\lambda\in H^{1,r}\, : \, u=\phi^\lambda+q\mathbb{G}^\lambda\right\},
\end{equation}
where $H^{1,r}=(1-\Delta)^{-1/2}L^r$ is the usual Sobolev space. 
As for the classical $H^{1}$ space, also in the presence of a point interaction we have the Sobolev embedding 
\begin{equation}\label{Sob-emb}
 H^{1}_{\alpha}\hookrightarrow L^{\rho} \hbox{ for any } \rho\geq2,   
\end{equation}
namely there exists a $C=C(\rho)>0$ such that $\|u\|_{L^{\rho}}\leq C\|u\|_{H^{1}_{\alpha}}$, provided $\lambda>|e_\alpha|$. Indeed, 
\[
\|u\|_{L^{\rho}}\lesssim \|\phi^\lambda\|_{L^{\rho}}+|q|\|\mathbb G^\lambda\|_{L^{\rho}}\lesssim \|\phi^\lambda\|_{H^1}+|q|\sim \|u\|_{H^1_{\alpha}}.
\]
Beside the above embedding, a  characterization property will be  crucially used in establishing a well-posedness theory. Precisely, we have the following (for a proof we refer to \cite{GR24}).
\begin{prop}\label{prop-space-equiv}
For any $r<2$, $H^{1,r}_{\alpha}=H^{1,r}$.
\end{prop}
\begin{remark}\label{rem-char}
    A similar characterization is  given for  more general singular spaces. Indeed, it holds that $H^{s,p}_{\alpha}(\mathbb R^d)=H^{s,p}(\mathbb R^d)$ provided $s<\frac{d}{p}-d+2$, $d=2,3$. As we aim at establishing an existence theory in the energy space, the characterizion can be used only in the two-dimensional  framework. 
\end{remark}

An essential tool in our ideas to establish the local well-posedness theory for \eqref{NLS} will be a  generalized logarithmic Hardy inequality in a critical Sobolev space. Precisely, we will make use of the following. 
\begin{lemma}
Let  $1<a,b<\infty$ such that $1+a-b<0$. Then we have that there exists a constant $C>0$ such that for any $u\in H^1$
\begin{equation}\label{log-hardy}
\int_{|x|<\frac12}\frac{|u(x)|^a}{|x|^2|\log(|x|)|^b}dx\leq C\|u\|^a_{H^1}.
\end{equation}
\end{lemma}
\noindent The proof of the above result is due to Machihara, Ozawa, and Wadade, see \cite{MOW2013}. 
\begin{remark}
The previous Lemma is specifically tailored for our scope, but in its full generality holds true also in the $n$-dimensional Euclidean space when the critical Sobolev space appearing in the right-hand side  of \eqref{log-hardy} is replaced by the critical Sobolev-Lorentz space $H^{p/n}_{p,q}(\mathbb R^n):=(1-\Delta)^{-p/n}L^{p,q}(\mathbb R^n)$, where $L^{p,q}$ is the classical Lorentz space, provided some conditions on the parameters $a,b$ (also depending  on $q$) are satisfied. We refer the reader to \cite[Theorem 1.1]{MOW2013}.
\end{remark}
\begin{remark}
  It is worth mentioning that for an integer $s$, the Sobolov-Lorentz space $H^s_{p,q}$ coincides with the usual Sobolev space $H^s$ provided $p=q=2$. See \cite[Theorem 6.2.3]{BL76}.
\end{remark}

We conclude this subsection by recalling the non-endpoint  Leibniz rule which will be used later on. 
\begin{lemma} Let $1<m_j\leq\infty$ for $j=1,2,3,4$ such that $\frac1r=\frac{1}{m_1}+\frac{1}{m_2}=\frac{1}{m_3}+\frac{1}{m_4}$, with $\frac12<r<\infty$. Then
\begin{equation}\label{eq:leibniz}
\|fg\|_{H^{1,r}}\lesssim \|f\|_{H^{1,m_1}}\|g\|_{L^{m_2}}+\|f\|_{L^{m_3}}\|g\|_{H^{1,m_4}}.
\end{equation}
\end{lemma}
\noindent A proof of the Lemma above can be found in \cite{GO2014}, or in the monograph \cite{MuSc} with a different approach.

\subsection{Ground states} In this subsection we recall an existence result for the ground states associated to \eqref{NLS}, and a few  properties  of the Pohozaev functional. In what follows, we consider the focusing equation, namely the nonlinearity in the right-hand side of \eqref{NLS} is $-|u|^{p-1}u$. \medskip

Let us consider the action functional 
\begin{equation}\label{def:action}
S_\omega(u)=E(u)+\frac{\omega}{2}M(u),
\end{equation}
where $E$ and $M$ are energy and the mass of a function $u\in H^1_{\alpha}$ as defined in \eqref{eq:energy} and \eqref{eq:mass}, respectively. We recall that a standing wave is a solution of \eqref{NLS} of the form $u(t,x)=e^{i\omega t}v(x)$, hence $v=v(x)$ satisfies  the  stationary equation 
\begin{equation}\label{eq:stationary}
\Delta_{\alpha}v+\omega v-|v|^{p-1}v=0.
\end{equation}
The set of ground states is defined instead as the set of solution to \eqref{eq:stationary} minimizing the action functional $S_\omega$, i.e.
\begin{equation}
G=\{v_\omega \in H^1_\alpha \hbox{ s.t. } S_\omega(v_\omega)\leq S_\omega(v) \hbox{ for any } v\in H^1_\alpha \hbox{ and } v_\omega \hbox{ satisfies } \eqref{eq:stationary}\}.
\end{equation}
Existence of solution to \eqref{eq:stationary} was proved in \cite{FGI22}, via minimization of the functional  $S_\omega$ constrained to the Nehari manifold 
\[
N_\omega(u):= \{u \in H^1_\alpha  \, \hbox{ s.t. } \,  F (u)+\omega\|u\|_{L^2}^2-\|u\|^{p+1}_{L^{p+1}}=0\}.
\]
Specifically, it was proved in \cite{FGI22} that for any $\omega>|e_\alpha|$ a ground state $v_\omega$ exists,  it minimizes the following variational problem:
\[
\mathfrak i_\omega=\inf_{\{0\neq v\in  N_\omega \}} S_\omega(v),
\]
and in particular $\mathfrak i_\omega=S_\omega(v_\omega)$. We also refer to the work of Adami, Boni, Carlone, and Tentarelli \cite{ABCT}.\smallskip

Via standard variational argument, see \cite{FN2023}, we have that the set of initial data $u_0\in H^1_{\alpha}$ satisfying $E(u_0)\geq0$, $S_{\omega}(u_0)<S_{\omega}(v_{\omega})$, and $P(u_0)<0$, where
\[
P(u)=F(u)-\frac{p-1}{p+1}\|u\|_{L^{p+1}}^{p+1}+\frac{|q|^2}{4\pi}, \quad u\in H^1_{\alpha},
\]
is the Pohozaev functional, is invariant along the flow of \eqref{NLS}. More precisely, we have the following.
\begin{prop}\label{prop:P-neg}
Let $u_0\in H^1_{\alpha}$ an initial datum fulfilling    $S_{\omega}(u_0)<S_{\omega}(v_{\omega})$, $E(u_0)\geq0$,  and $P(u_0)<0$. Then the solution $u(t) $ to the Cauchy problem \eqref{NLS},  preserves the same bounds for any $t\in(-T_{\min}, T_{\max})$, and in particular
\begin{equation}\label{P-bound}
P(u(t))< 2(S_{\omega}(u_0)-S_{\omega}(v_{\omega}))<0.
\end{equation}
\end{prop}
\begin{remark}
The bound involving $S_\omega$ and the one on $E$ are straightforward by conservation of mass and energy, while the one $P$ follows by standard variational analysis.  
\end{remark}

\section{LWP theory}\label{sec:LWP}
We prove here the local existence result \emph{\`a la Kato} for \eqref{NLS}. As discussed earlier, a local existence theory for \eqref{NLS} was given in \cite{FGI22} by implementing the abstract theory of \cite{OSY12}. For the reasons explained in the Introduction, we give a direct proof using the Kato theory. \medskip

In the following we denote $g(u)=\pm|u|^{p-1}u$. From \cite{Cazenave} we can claim that in order to apply the Kato's theory, it is enough to prove that for some $2\leq r, \rho<\infty$ 
\begin{equation}\label{kato1}
\|g(u)-g(v)\|_{L^{r'}}\leq C(M)\|u-v\|_{L^\rho}
\end{equation}
for all $u,v\in H^1_{\alpha}$ such that $\|u\|_{H^1_{\alpha}},\|v\|_{H^1_{\alpha}}\leq M$ and that 
\begin{equation}\label{kato2}
\|g(u)\|_{H^{1,r'}_{\alpha}}\leq C(M)(1+\|u\|_{H^{1,\rho}_{\alpha}})
\end{equation}
for all $u\in H^1_{\alpha}\cap H^{1,r}_{\alpha}$ such that $\|u\|_{H^1_{\alpha}}\leq M$.
\\

Recall the estimate $||u|^{p-1}u-|v|^{p-1}v|\lesssim |u-v|(|u|^{p-1}+|v|^{p-1})$. Take $r'=2-\epsilon$, $\epsilon\in(0,1)$ then by the H\"older inequality with $\frac{1}{\tilde q}+\frac12=\frac{1}{2-\epsilon}$ and the Sobolev \eqref{Sob-emb}, we have
\begin{equation}
\begin{aligned}
\||u|^{p-1}u-|v|^{p-1}v\|_{L^{2-\epsilon}}&\lesssim \|u-v\|_{L^2}(\|u\|^{p-1}_{L^{(p-1)\tilde q}}+\|v\|^{p-1}_{L^{(p-1)\tilde q}})\\
&\lesssim\|u-v\|_{L^2}(\|u\|^{p-1}_{H^1_{\alpha}}+\|v\|^{p-1}_{H^1_{\alpha}}).
\end{aligned}
\end{equation}
Therefore, we have that \eqref{kato1} is verified.\\

We have from Proposition \ref{prop-space-equiv} that, for $\sigma<2$, $\|f\|_{H^{1,\sigma}_\alpha}\sim\|f\|_{H^{1,\sigma}}$. Hence our first idea is to reduce our task to estimate the $H^{1,\sigma}$-norms for $\sigma$'s less than 2.  \\
As before, take $r'=2-\epsilon$, where $\epsilon\in(0,1)$.  Let us rewrite 
\begin{equation}\label{eq:deco-u}
|u|^{p-1}u=\psi |u|^{p-1}u+(1-\psi)|u|^{p-1}u
\end{equation}
where $\psi\in C^\infty_c$ such that $\psi=1$ in a neighbourhood  of $x=0$ and its support is contained in the ball $|x|\leq1/2$. Note that if $u\in H^1_{\alpha}$ then $(1-\psi)u\in H^1$ since the singularity at the origin is removed.  \medskip

Let us focus on the first term $\psi |u|^{p-1}u$. Consider a further cut-off function $\psi_1$ such that $\psi_1=1$ on the support of $\psi$. Then, $\psi|u|^{p-1}u=\psi|\psi_1u|^{p-1}(\psi_1u)$. Since now on along  this section, we omit the $\lambda$-dependence of $\phi^\lambda$ and $\mathbb G^{\lambda}$.  Hence, in a neighbourhood of the origin 
\begin{equation}\label{loc-u-origin}
\begin{aligned}
\psi |u|^{p-1}u&=\psi(\psi_1\phi+q\psi_1 \mathbb G)|\psi_1\phi+q\psi_1 \mathbb G|^{p-1}\\
&\sim\psi(\psi_1\phi+q\psi_1 \log|x|)|\psi_1\phi+q\psi_1 \log|x||^{p-1}\\
&=\psi(\psi_1\phi-q\psi_1 \log(|x|^{-1}))|\psi_1\phi-q\psi_1 \log(|x|^{-1})|^{p-1}\\
&=\psi\log^p(|x|^{-1})\left(\frac{\psi_1\phi}{\log(|x|^{-1})}-q\psi_1 \right)\left|\frac{\psi_1\phi}{\log(|x|^{-1})}-q\psi_1 \right|^{p-1}.
\end{aligned}
\end{equation}
 In the last step, we eliminated the $\log(|x|)$-type singularity from the localized non-regular term $\mathbb G$ by placing it as a factor of the regular term $\psi_1\phi$. Later, we will demonstrate how to manage this new term with $\log(|x|)$ incorporated into the regular part of the solution.\\

The function $h:=\frac{\psi_1\phi}{\log(|x|^{-1})}$ is clearly in $L^2$ as it is supported in $|x|<1/2$ and in that ball $\frac{1}{\log(|x|^{-1})}\leq \frac{1}{\log2}$. We claim that actually $h\in H^1$. Straightforward calculations give
\[
\nabla h=-\frac{\psi_1\phi }{|x|^2\log^2(|x|^{-1})}x+\frac{1}{\log(|x|^{-1})}
\nabla(\psi_1\phi),
\]
then
\[
|\nabla h|\lesssim \frac{|\psi_1\phi|}{|x|\log^2(|x|^{-1})}+\frac{|\nabla(\psi_1\phi)|}{\log(|x|^{-1})}.
\]

\noindent At this point we invoke the logarithmic Hardy
inequality \eqref{log-hardy} with $a=2$, $b=4$, so that 
\[
\int_{|x|<\frac12}\frac{|\psi_1\phi|^2}{|x|^2\log^{4}(|x|^{-1})}dx=\int_{|x|<\frac12}\frac{|\phi|^2}{|x|^2\log^{4}(|x|^{-1})}dx\lesssim \|\phi\|_{H^1}^2.
\]
\noindent Using the  boundedness of $\frac{1}{\log(|x|^{-1})}$ and the properties of $\psi_1$, we control the second term 
\[
\int_{|x|<\frac12}\frac{|\nabla(\psi_1\phi)|^2}{\log^{2}(|x|^{-1})}dx\lesssim\int_{|x|<\frac12}|\nabla(\psi_1\phi)|^2 dx\lesssim\|\phi\|^2_{H^1}.
\]
\\
\noindent At this point, we introduce the function 
\[\tilde u=\frac{\psi_1\phi}{\log(|x|^{-1})}-q\psi_1.
\]
By means of the Leibniz rule \eqref{eq:leibniz}, we estimate \eqref{loc-u-origin} as follows:
\begin{align}\label{kato-staff1}
\|\psi\log^p(|x|^{-1})\tilde u |\tilde u|^{p-1}\|_{H^{1,r}}&\lesssim \| (\psi\log^p(|x|^{-1}))\|_{H^{1,m_1}}\|\tilde u\|_{L^{pm_2}}^{p}\\\label{kato-staff2}
&+\| \psi\log^p(|x|^{-1})\|_{L^{m_3}}\| (\tilde u|\tilde u|^{p-1})\|_{H^{1, m_4}},
\end{align}
where $r=2-\epsilon$, $\epsilon\in(0,1)$, and
\[
\frac{1}{2-\epsilon}=\frac{1}{r}=\frac{1}{m_1}+\frac{1}{m_2}=\frac{1}{m_3}+\frac{1}{m_4}.
\]\\
\noindent Let  $m_1=2-\frac{\epsilon}{2}$ and  $m_2=\frac{(4-\epsilon)(2-\epsilon)}{\epsilon} $. We have that $\psi \log^p(|x|^{-1})\in L^{m_1}$ and $\nabla \psi \log^p(|x|^{-1})\in L^{m_1}$, indeed $\psi$ and  $\nabla \psi$ are  supported in the ball $|x|<1/2$, and $\log^p(|x|^{-1})\lesssim |x|^{-\beta}$ for any $\beta>0$ in that ball. In particular it holds for $\beta<(4-\epsilon)/\epsilon$ which gives the desired integrability. Moreover, basic calculus gives    $\displaystyle \psi\nabla\log^p(|x|^{-1})=-p\psi\log^{p-1}(|x|^{-1})\frac{x}{|x|^2}$,
then 
\[
|\psi\nabla\log^p(|x|^{-1})|\lesssim \psi\frac{1}{|x|}\log^{p-1}(|x|^{-1}),
\]
which is in $L^{m_1}$ as well for the same consideration as above. Since $\tilde u\in H^1$  and by the Sobolev embedding $H^1\subset L^{\rho}$ for any $\rho\geq 2$, then $\tilde u \in L^{m_2p}$ for  $m_2p<\infty$. Hence, the right-hand side of \eqref{kato-staff1} is estimated by 
\begin{equation}\label{kato-staff5}
\| \psi\log^p(|x|^{-1})\|_{H^{1,m_1}}\|\tilde u\|_{L^{pm_2}}^{p} \lesssim \|\tilde u\|^p_{H^1}.
\end{equation}
As above, take $m_4=2-\frac{\epsilon}{2}$ and  $m_3=\frac{(4-\epsilon)(2-\epsilon)}{\epsilon} $. By the Kato-Staffilani estimate, \cite{Kato95, Staffilani95}, and the Sobolev embedding, we have
\[
\|\nabla (\tilde u|\tilde u|^{p-1})\|_{L^{m_4}}\lesssim \|\tilde u\|^{p-1}_{L^{2(p-1)(4-\epsilon)/\epsilon}}\|\tilde u\|_{H^1}\lesssim \|\tilde u\|^p_{H^1}.
\]
Note that $\tilde u$ is in $H^1$  since  $\tilde u=h-q\psi_1$ (by definition) and by the previous discussion both $h$ and $\psi_1$  belong to $ H^1$. The $L^{m_3}$-norm of $\psi\log^p(|x|^{-1})$ is finite, then for \eqref{kato-staff2} we have 
\begin{equation}\label{kato-staff6}
\| \psi\log^p(|x|^{-1})\|_{L^{m_3}}\| (\tilde u|\tilde u|^{p-1})\|_{H^{1, m_4}}\lesssim \|\tilde u\|^p_{H^1}.
\end{equation}

\noindent The observation now is that 
\begin{equation}\label{kato-staff3}
\|\tilde u\|_{H^1}\lesssim \|\phi\|_{H^1}+|q|\sim \|u\|_{H^1_{\alpha}}.
\end{equation}
Therefore \eqref{kato-staff5} and \eqref{kato-staff6} jointly with \eqref{kato-staff3} along with the definition \eqref{loc-u-origin} yield
\[
\|\psi|u|^{p-1}u\|_{H^{1,r}}\lesssim\|\psi\log^p(|x|^{-1})\tilde u |\tilde u|^{p-1}\|_{H^{1,r}}\lesssim \|u\|^p_{H^1_{\alpha}}.
\]
Note that the constants hidden in the notation $\lesssim$ and $\sim$ only depends on $\psi_1$. \\

At this point we move to the estimate of the remaining term in \eqref{eq:deco-u}, namely $(1-\psi)|u|^{p-1}u$. Observe that away from the origin the function $\mathbb{G}$ is regular and decays at infinity faster than exponentially, see \eqref{eq.asymptotic zero}, hence
\[
||u|^{p-1}u|= |\phi+q \mathbb{G}|^{p}\lesssim |\phi|^p+|\mathbb{G}|^p\in L^{2-\epsilon}.
\]
Moreover, as $|\nabla(|u|^{p-1}u)|\lesssim |u|^{p-1}|\nabla u|,$ by the H\"older inequality
\[
\||u|^{p-1}|\nabla u|\|_{L^{2-\epsilon}}^{2-\epsilon}\leq\|\nabla u\|_{L^2}^{2-\epsilon}\|u\|^{(p-1)(2-\epsilon)}_{L^{2(p-1)(2-\epsilon)\epsilon}}\lesssim \|u\|^{p(2-\epsilon)}_{H^1},
\]
hence also $|u|^{p-1}\nabla u\in L^{2-\epsilon}$. 
We then conclude with the fact that $(1-\psi)|u|^{p-1}u\in H_{\alpha}^{1,r}$, $r=2-\epsilon$.
\begin{proof}[Proof of Theorem \ref{thm-lwp}] The proof of estimates \eqref{kato1} and \eqref{kato2} gives Theorem \ref{thm-lwp} along the same lines of the theory developed in the Cazenave's monograph \cite[Chapter 4.4]{Cazenave}.
\end{proof}

\section{Blow-up}\label{Sec-blowup}

We move now to the proof of the dynamical results of radial solutions. The approach is based on virial estimates. Comparing to the classical case in the presence of the unperturbed Laplacian,  virial identities for \eqref{NLS} must take account of the singular nature of the operator $\Delta_{\alpha}$, hence we cannot straightforwardly proceed as for the equation with a free Laplacian.  To prepare the rigorous proof of the virial identity we note that we crucially have from \eqref{def-G} the following identity which holds in distributional sense   for any regular function $\eta$ with $\eta(0)=0$:
\begin{equation}\label{id-distr}
\eta (\lambda - \Delta) \mathbb{G}^\lambda =0.
\end{equation}
\begin{lemma}\label{ble5}
If $\eta$ a smooth cut-off function which is $|x|^2$ for $x$ near the origin.
Then for any $u \in H^2_\alpha$
we have:\vspace{2mm}
    
\noindent   \textit{(i)} $\eta u \in H^2$; \medskip
     
\noindent     \textit{(ii)} $\Delta_{\alpha} (\eta u) = -\Delta (\eta u)$.
\end{lemma}
\begin{proof}
Let us suppose that $i)$ holds true. Then by definition of $H^2_\alpha$ as in \eqref{def-H2-alpha}, along  with the decomposition into the regular and irregular part of $u$ as in \eqref{def-equiv}, we have that 
\[
(\lambda-\Delta_{\alpha})(\eta u)=(\lambda-\Delta)(\eta u),
\]
indeed $\eta u$ only consists of a regular part. The thesis follows.  
Let us now prove that $i)$ is verifies. It is enough to show that the multiplication by $\eta$ regularizes $\mathbb{G}^\lambda$, namely $
\eta \mathbb{G}^\lambda \in H^2.$ By \eqref{eq.asymptotic zero}, the singularity of $\mathbb{G}^\lambda$ is of the type $\log|x|$ as $x\to0$, so $\eta \mathbb{G}^\lambda\sim |x|^2\log|x|:=h(x)$ as $|x|\to0$.  
Therefore, straightforward calculations give that in a neighbourhood of $x=0$ the function $h$ is clearly in $L^2$; moreover $|\nabla h|\lesssim |x|(\log|x|+1)$ which is in $L^2$, and $|\nabla^2h|\lesssim \log{(|x|^{-1})}$ which is in $L^2_{\loc}$ as well. The proof is concluded.
\end{proof}

With the regularity Lemma above, we can state the virial identities used along the paper. 
Consider the virial functional 
\begin{equation}\label{vir}
V(t)=\int \eta(x)|u(t)|^2dx
\end{equation}
{which is well defined provided $\eta$ is a compactly supported smooth cut-off function}. In what follows, we omit the time-dependence of $u$, $\phi^\lambda$, and $q$ when no confusion may arise. 
\begin{prop} Let $u\in H^1_\alpha$ be a radial solution to \eqref{NLS}, and $\eta$ a smooth cut-off function  having compact support and which is $\frac{|x|^2}{2}$ for $x$ in a neighbourhood of the origin. Then we have 
    \begin{equation}\label{vir-prime}
    \frac{d}{dt}V(t)=2 \mathrm{Im} \int\nabla\eta\cdot\nabla u\bar udx
\end{equation}
and 
\begin{equation}\label{vir-second}
    \frac{d^2}{dt^2}V(t)=4P(u(t))+\mathcal R(u(t)),
\end{equation}
where $\mathcal R(u(t))$ is a remainder term defined  in \eqref{def-remainder} below.
\end{prop}
\begin{proof}
 Using a density argument we approximate $u$ by a sequence in $H^2_\alpha$. Therefore, we can assume $u\in H^2_\alpha$ and then by  using the self-adjointness of $ \Delta_{\alpha}$ and Lemma \ref{ble5}, we get 
\begin{equation}
    \begin{aligned}
      \frac{d}{dt}V(t) &= 2\mathrm{Re}\int \eta \partial_t u  \bar{u}dx 
      =2\mathrm{Re}\int \left(-i\eta  \Delta_{\alpha} u  \right)   \bar{u}dx \\
      &=  2\mathrm{Im}\int\eta \bar{u} \Delta_{\alpha} u dx
     =   -2\mathrm{Im}\int u  \Delta (\eta \bar{u})dx = 2 \mathrm{Im}\int \bar{u}  \nabla \eta \nabla udx,
    \end{aligned}
\end{equation}
which is \eqref{vir-prime}. Note that we do not use the behaviour of $\eta$ in this step.\smallskip

As for the second derivative of $V(t)$,   by using again the equation \eqref{NLS} solved by $u$, differentiating \eqref{vir-prime} in time yields to
\begin{align}\label{vir-2a0}
\frac{d^2}{dt^2}V(t)&= 2\rea \int\Delta\eta u\Delta_{\alpha}\bar{u}dx\\\label{vir-2b}
&+4\rea \int\Delta_{\alpha}\bar{u}\nabla u\cdot\nabla\eta dx\\\label{vir-2a}
&- 2 \rea \int\Delta\eta |u|^{p+1} dx\\\label{vir-2c}
&-4\rea\int |u|^{p-1}\bar u \nabla u\cdot\nabla\eta dx.
\end{align}
We expand   \eqref{vir-2c} obtaining 
\[
\eqref{vir-2c}=\frac{4}{p+1}\int(\Delta\eta-2)|u|^{p+1}dx+\frac{8}{p+1}\int|u|^{p+1}dx
\]
and we rewrite
\[
 \eqref{vir-2a}=-2\int(\Delta\eta-2)|u|^{p+1}dx-4\int|u|^{p+1}dx;
\]
hence
\[
\eqref{vir-2a}+\eqref{vir-2c}=-\frac{2(p-1)}{p+1}\int(\Delta\eta-2)|u|^{p+1}dx-\frac{4(p-1)}{p+1}\int |u|^{p+1}dx.
\]
To deal with \eqref{vir-2b} we explicitly  write the solution $u=\phi^\lambda+q\mathbb{G}^\lambda$, and the action of the operator $\Delta_{\alpha}u=-\Delta \phi^\lambda-q\lambda \mathbb{G}^\lambda$, then we get 
\begin{align}
\eqref{vir-2b}&=4\rea\int\nabla\eta\left(-\Delta \bar\phi^\lambda-\bar q\lambda \mathbb{G}^\lambda \right)\left(\nabla\phi^\lambda+q\nabla \mathbb{G}^\lambda\right)dx\\
&=4\rea\int\nabla\eta\left(-\Delta \bar\phi^\lambda\nabla\phi^\lambda-q\Delta\bar\phi^\lambda\nabla \mathbb{G}^\lambda-\bar q\lambda \mathbb{G}^\lambda\nabla\phi^\lambda-\lambda|q|^2\mathbb{G}^\lambda\nabla \mathbb{G}^\lambda \right)\\
&= A+ B+ C+ D.
\end{align}
We have, in order:
\begin{equation}
\begin{aligned}
A&=4\rea\int\nabla^2\eta\nabla\bar\phi^\lambda\nabla\phi^\lambda dx-2\int\Delta\eta|\nabla\phi^\lambda|^2dx,\\
B+ C&=4\rea\int \bar q\lambda\phi^\lambda \di(\nabla\eta \mathbb{G}^\lambda) -q\bar\phi^\lambda\Delta(\nabla\eta\nabla \mathbb{G}^\lambda)dx,\\
D&=2\lambda|q|^2\int\Delta\eta(\mathbb{G}^\lambda)^2dx.
\end{aligned}
\end{equation}
In conclusion,
\begin{align}
\eqref{vir-2b}&=4\rea\int\nabla^2\eta\nabla\bar\phi^\lambda\nabla\phi^\lambda dx-2\int\Delta\eta|\nabla\phi^\lambda|^2dx\\
&+4\rea\int \bar q\lambda\phi^\lambda \di(\nabla\eta \mathbb{G}^\lambda) -q\bar\phi^\lambda\Delta(\nabla\eta\nabla \mathbb{G}^\lambda)dx\\
&+2\lambda|q|^2\int(\Delta\eta-2)(\mathbb{G}^\lambda)^2dx\\
&+4\lambda|q|^2\int(\mathbb{G}^\lambda)^2dx.
\end{align}
Let us expand the  term in the right-hand side of \eqref{vir-2a0}.
\begin{align}
\hbox{r.h.s.}\eqref{vir-2a0}&=4\int u \Delta_{\alpha}\bar{u}dx+2\rea\int(\Delta\eta-2)u\Delta_{\alpha}\bar{u}dx\\
&=4 F(u) +2\rea\int(\Delta\eta-2)u\Delta_{\alpha}\bar{u}dx.
\end{align}
Note that 
$ \displaystyle 4\lambda|q|^2\int(\mathbb{G}^\lambda)^2dx=\frac{|q|^2}{\pi},$
and by gluing all contribution together, recalling that $\displaystyle F(u)=2E(u)+\frac{2}{p+1}\|u\|_{L^{p+1}}^{p+1}$,  we obtain
\begin{align}\label{vir-Q}
\frac{d^2}{dt^2}V(t)&=8E(u)-\frac{4(p-3)}{p+1}\|u\|_{L^{p+1}}^{p+1}+\frac{|q|^2}{\pi}\\\label{vir-p+1}
&-\frac{2(p-1)}{p+1}\int(\Delta \eta-2)|u|^{p+1}dx\\\label{vir-uHu}
&+2\rea\int(\Delta\eta-2)u\Delta_{\alpha}\bar{u}dx\\\label{vir-nega}
&+4\rea\int\nabla^2\eta\nabla\bar\phi^\lambda\nabla\phi^\lambda dx-2\int\Delta\eta|\nabla\phi^\lambda|^2dx\\\label{vir-non-lo}
&+4\rea\int \bar q\lambda\phi^\lambda \di(\nabla\eta \mathbb{G}^\lambda) -q\bar\phi^\lambda\Delta(\nabla\eta\nabla \mathbb{G}^\lambda)dx\\\label{vir-G-lam}
&+2\lambda|q|^2\int(\Delta\eta-2)(\mathbb{G}^\lambda)^2dx.
\end{align}
Observe that right-hand side $\eqref{vir-Q}=4P(u(t))$, and the proof of \eqref{vir-second} is done by defining 
\begin{equation}\label{def-remainder}
\mathcal R(u(t))=\eqref{vir-p+1}+\eqref{vir-uHu}+\eqref{vir-nega}+\eqref{vir-non-lo}+\eqref{vir-G-lam}.
\end{equation}
\end{proof}

At this point we precisely describe the function $\eta$ we are going to use to estimate $R(u(t)).$ Let $\theta: [0,\infty) \to [0,1]$ be a smooth  function satisfying
\begin{align} \label{def-theta}
\theta(\tau)= \left\{
\begin{array}{ccl}
1 &\text{if} &0\leq \tau\leq 1, \\
0 &\text{if} & \tau\geq 2.
\end{array}
\right.
\end{align}
We define the function $\Theta: [0,\infty) \to [0,\infty)$ by
\[
\Theta(|x|):= \int_0^{|x|} \int_0^s \theta(\tau) d\tau ds.
\]
For $R>0$, we define the radial function $\eta: \R^2 \to \R$ by
\begin{align} \label{phi-R-rad}
\eta = \eta_R(|x|) := R^2 \Theta(|x|/R).
\end{align}
With such a cut-off  function we have the following. 
\begin{prop}
Provided $3<p\leq 5$, for sufficiently large $R$, the remainder $\mathcal R$ is estimated by 
\begin{equation}
\mathcal R(u(t))\lesssim o_R(1) + o_R(1)\|u(t)\|_{H^1_\alpha}^2,
\end{equation}
where $o_R(1)$ is uniform in time. Consequently, 
\begin{equation}\label{blow-est}
    \frac{d^2}{dt^2}V(t)
\lesssim 4P(u(t))+o_R(1)+o_R(1)\|u(t)\|_{H^1_\alpha}^2.
\end{equation}
\end{prop}

\begin{remark}
With standard notation, $o_R(1)$ stands for a function depending on $R$ and independent of time such that $\displaystyle\lim_{R\to+\infty}o_R(1)=0$.
\end{remark}
\begin{proof}
We use  the Strauss decay estimate for radial function to estimate the contribution given by \eqref{vir-p+1}. To this end, let us consider a radial bump function $\chi$ having the following properties:
\[
\chi\in C^{\infty}_c, \quad 0\leq\chi\leq1, \quad \chi=1 \quad\hbox{ for }\quad |x|\geq \frac{R}{2}, \quad \chi=0 \quad\hbox{ for }\quad |x|\leq\frac{R}{4}.
\]
By means of the conservation of the mass, the properties of $\chi$, the fact that $\chi u\in H^1$, the decomposition  of the function $u=\phi^\lambda+q\mathbb{G}^\lambda$, and the decay properties of $\mathbb{G}^\lambda$ away from the origin, we estimate
\begin{align}
\eqref{vir-p+1}&\lesssim\int_{|x|\geq R }|u|^{p+1}dx\lesssim\int_{|x|\geq \frac{R}{4}}\left(\frac{|x|^{1/2}|\chi u|}{|x|^{1/2}}\right)^{p-1}|u|^2dx\\
&\lesssim R^{-\frac{p-1}{2}}\left(\sup_{|x|\geq \frac{R}{4}}|x|^{1/2}|\chi u(x)|\right)^{p-1}\lesssim R^{-\frac{p-1}{2}}\|\chi u\|_{L^2}^{\frac{p-1}{2}}\|\nabla (\chi u)\|_{L^2}^{\frac{p-1}{2}}\\
&\lesssim R^{-\frac{p-1}{2}}\|\nabla (\chi u)\|_{L^2}^{\frac{p-1}{2}}\\
&\lesssim R^{-\frac{p-1}{2}}\left( \|\nabla \chi\|_{L^\infty}\|u\|_{L^2}+\|\chi\|_{L^\infty}\|\nabla\phi^\lambda+q\nabla \mathbb{G}^\lambda\|_{L^2(|x|\geq\frac{R}{4})}\right)^{\frac{p-1}{2}}\\
&\lesssim R^{-\frac{p-1}{2}}\left( M^{1/2}+\|\nabla\phi^\lambda\|_{L^2}+|q|\|\nabla \mathbb{G}^\lambda\|_{L^2(|x|\geq\frac{R}{4})}\right)^{\frac{p-1}{2}}\\
&\lesssim o_R(1)+o_{R}(1)(\|\nabla\phi^\lambda\|_{L^2}+|q|)^{\frac{p-1}{2}}\\
&\lesssim o_R(1)+o_{R}(1)\|u\|_{H^1_\alpha}^{\frac{p-1}{2}}.
\end{align}

\noindent Let us deal with \eqref{vir-uHu}.
\begin{align}
\eqref{vir-uHu}&=-2\rea\int(\Delta\eta-2)(\phi^\lambda+q\mathbb{G}^\lambda)(\Delta \phi^\lambda+\bar q\lambda \mathbb{G}^\lambda)dx\\
&=-\int(\Delta\eta-2)\left(\Delta|\phi^\lambda|^2-2|\nabla\phi^\lambda|^2+\bar q \lambda\phi^\lambda \mathbb{G}^\lambda+q\mathbb{G}^\lambda\Delta\bar\phi^\lambda+\lambda|q|^2(\mathbb{G}^\lambda)^2\right)dx\\
&=-\int\Delta^2\eta|\phi^\lambda|^2dx+2\rea\int(\Delta\eta-2)|\nabla\phi^\lambda|^2dx\\
&-2\rea\int(\Delta\eta-2)\left( \bar q \lambda\phi^\lambda \mathbb{G}^\lambda+q\mathbb{G}^\lambda\Delta\bar\phi^\lambda+\lambda|q|^2(\mathbb{G}^\lambda)^2 \right)dx.
\end{align}
Gluing all together we end-up with 
\begin{align}\label{vir-neg1}
\frac{d^2}{dt^2}V(t)&\lesssim 4P(u(t))+o_R(1)\int|\phi^\lambda|^2dx+R^{-\frac{p-1}{2}}\|u(t)\|_{H^1_\alpha}^{\frac{p-1}{2}}\\\label{vir-neg2}
&+4\int(\eta''-1)|\nabla\phi^\lambda|^2dx\\\label{vir-neg3}
&-2\rea\int(\Delta\eta-2)\left( \bar q \lambda\phi^\lambda \mathbb{G}^\lambda+q\mathbb{G}^\lambda\Delta\bar\phi^\lambda\right)dx\\\label{vir-nonlocal}
&+4\rea\int \bar q\lambda\phi^\lambda \di(\nabla\eta \mathbb{G}^\lambda) -q\bar\phi^\lambda\Delta(\nabla\eta\nabla \mathbb{G}^\lambda)dx.
\end{align}
The term \eqref{vir-neg2} is negative as $\eta''\leq 1$, so it is simply estimated by zero. 
As for the term \eqref{vir-neg3} we note that the support of $\Delta\eta-2$ is contained outside the ball centered at the origin and radius $R$, so by using the $L^2$ integrability of $\mathbb{G}^\lambda$ outside a ball centered at the origin, we control by the Cauchy-Schwarz inequality and by the dominated convergence theorem
\[
\rea\int(\Delta\eta-2)\bar q \lambda\phi^\lambda \mathbb{G}^\lambda dx\lesssim |q|\|\phi^\lambda\|_{L^2}\|\mathbb{G}^\lambda\|_{L^2(|x|\geq R)}\lesssim o_R(1)(|q|^2+\|\phi^\lambda\|_{L^2}^2).
\]
Similarly, after integration  by part, we estimate
\begin{align}
\rea\int(\Delta\eta-2)q\mathbb{G}^\lambda\Delta\bar\phi^\lambda dx&=-q\rea\left(\int\nabla(\Delta\eta-2)\mathbb{G}^\lambda\nabla\bar\phi^\lambda dx+(\Delta\eta-2)\nabla \mathbb{G}^\lambda\nabla\bar\phi^\lambda dx\right)\\
&\lesssim o_R(1)(|q|^2+\|\nabla\phi^\lambda\|_{L^2}^2).
\end{align}
Therefore, 
\begin{equation}
 \eqref{vir-neg3}\lesssim o_R(1)(|q|^2+\|\phi^\lambda\|_{L^2}^2+\|\nabla\phi^\lambda\|_{L^2}^2).
\end{equation}
As for \eqref{vir-nonlocal}, we manipulate the integral as 
\begin{align}
\eqref{vir-nonlocal}&=4\rea\int \bar q\lambda\phi^\lambda \di((\nabla\eta-x) \mathbb{G}^\lambda) -q\bar\phi^\lambda\Delta((\nabla\eta-x)\nabla \mathbb{G}^\lambda)dx\\
&+4\rea\int \bar q\lambda\phi^\lambda \di(x \mathbb{G}^\lambda) -q\bar\phi^\lambda\Delta(x\nabla \mathbb{G}^\lambda)dx\\
&=4\rea\int \bar q\lambda\phi^\lambda \di((\nabla\eta-x) \mathbb{G}^\lambda) -q\bar\phi^\lambda\Delta((\nabla\eta-x)\nabla \mathbb{G}^\lambda)dx\\\label{vir-neg4}
&=-4\rea\int \bar q\lambda\nabla\phi^\lambda ((\nabla\eta-x) \mathbb{G}^\lambda) -q\nabla\bar\phi^\lambda\nabla((\nabla\eta-x)\nabla \mathbb{G}^\lambda)dx,
\end{align}
where we used the identity 
\[
 \rea\int \bar q\lambda\phi^\lambda \di(x \mathbb{G}^\lambda) -q\bar\phi^\lambda\Delta(x\nabla \mathbb{G}^\lambda)dx=0.
\]
By observing that the function $\nabla\eta-x$ is supported outside a ball of radius $R$ centered at the origin, similarly to the estimate for \eqref{vir-neg3} we get 
\begin{equation}
\eqref{vir-neg4}\lesssim  o_R(1)(|q|^2+\|\nabla\phi^\lambda\|_{L^2}^2).
\end{equation}
In the end, provided $p<5$, by using the Young inequality, we get
\begin{align}
\frac{d^2}{dt^2}V(t)&\lesssim 4P(u(t))+o_R(1)(|q|^2+\|\phi^\lambda\|_{L^2}^2+\|\nabla\phi^\lambda\|_{L^2}^2)+R^{-\frac{p-1}{2}}\|u(t)\|_{H^1_\alpha}^{\frac{p-1}{2}}\\
&\lesssim 4P(u(t))+o_R(1)+o_R(1)\|u(t)\|_{H^1_\alpha}^2.
\end{align}
If instead $p=5$, we directly have the estimate without employing Young inequality. 
\end{proof}

We prove now a refinement of \eqref{P-bound} which will be used in the proof of existence of finite-time blowing-up solutions. 
\begin{lemma}
Under the hypothesis of Proposition \ref{prop:P-neg}, provided $3<p\leq5$, we have the following refinement of the control on $P(u(t))$: the exist two  constants $c_1,c_2>0$  independent of $t$ such that 
\begin{equation}\label{Q-control-ref}
P(u(t))+ c_1\|u(t)\|^2_{H^1_\alpha}\leq -c_2.
\end{equation}
\end{lemma}
\begin{proof}
 Consider $\varepsilon>0$  to be chosen later. First of all, we note the following identity:
\[
F(u)=\frac{2}{p-3}\left((p-1)E(u)-P(u)+\frac{|q|^2}{4\pi} \right).
\]
Then we have
\begin{equation}
\begin{aligned}
&\phantom{\,\quad}P(u)+\varepsilon\left( \|\nabla \phi^\lambda\|_{L^2}^2+\lambda\|\phi^\lambda\|_{L^2}^2+ |q|^2\right)\\
&=P(u)+\varepsilon\left(\|\nabla \phi^\lambda\|_{L^2}^2+\lambda\|\phi^\lambda\|_{L^2}^2\pm\lambda\|u\|_{L^2}^2\pm \Gamma^\lambda_\alpha|q|^2+|q|^2\right)\\
&=P(u)+\varepsilon\left(F(u)+\lambda M(u)-\Gamma^\lambda_\alpha|q|^2+|q|^2\right)\\
&\leq\left(1-\frac{2\varepsilon}{p-3}\right)P(u)+\frac{2\varepsilon(p-1)}{p-3}E(u)+\varepsilon\lambda M(u)+\varepsilon\left(\frac{1}{2\pi(p-3)}+1-\Gamma^\lambda_\alpha
\right)|q|^2,
\end{aligned}
\end{equation}
and by conservation of energy and mass we can estimate, provided $\varepsilon$ is small enough,
\[
P(u)+\varepsilon\left( \|\nabla \phi^\lambda\|_{L^2}^2+\lambda\|\phi^\lambda\|_{L^2}^2+ |q|^2\right)\leq -\frac{\delta}{2}+\varepsilon\left(\frac{1}{2\pi(p-3)}+1-\Gamma^\lambda_\alpha
 \right)|q|^2,
\]
where $\delta=2(S_{\omega}(v_\omega)-S(u_0))$.  Recall that  $\Gamma^\lambda_\alpha=\alpha+\frac{\gamma-\log2}{2\pi}+\frac{\log\lambda}{4\pi}$. Thus, for sufficiently large $\lambda$, possibly up to considering a smaller $\varepsilon$, we can infer that 
\[
P(u)+\varepsilon\left( \|\nabla \phi^\lambda\|_{L^2}^2+\lambda\|\phi^\lambda\|_{L^2}^2+ |q|^2\right)\leq -\frac{\delta}{2}
+\frac{\varepsilon}{2}|q|^2,
\]
and then
\[
P(u)+\frac\varepsilon2\left( \|\nabla \phi^\lambda\|_{L^2}^2+\lambda\|\phi^\lambda\|_{L^2}^2+ |q|^2\right)\leq -\frac{\delta}{2}.
\]
The claim follows with $c_1=\frac\varepsilon2$ and $c_2=\frac\delta2$.
\end{proof}

\begin{lemma}
Under the hypothesis of Proposition \ref{prop:P-neg}, provided $3<p\leq5$, we have 
\begin{equation}\label{blow-est-2}
\inf_{ t\in (-T_{\min},T_{\max}) }\|u(t)\|_{H^1_\alpha}^2\geq c>0.      
\end{equation}
\end{lemma}

\begin{proof}
Recall the definition 
$F(u)=\|\nabla\phi^\lambda\|_{L^2}^2+\lambda(\|\phi^\lambda\|_{L^2}^2-\|u\|_{L^2}^2)+\Gamma^\lambda_\alpha|q|^2.$
Note that 
\[
\left|\|u\|^2_{L^2}-\|\phi^\lambda\|^2_{L^2}\right|=\left||q|^2\|\mathbb G\|^2_{L^2}+2\rea\int\phi^\lambda \bar q\mathbb G dx\right|
\lesssim (\|\phi^\lambda\|^2_{L^2}+|q|^2).
\]
Then $|F(u)|\lesssim (\|\nabla\phi^\lambda\|^2_{L^2}+\|\phi^\lambda\|^2_{L^2}+|q|^2)\sim \|u\|^2_{H^1_{\alpha}}$,
while the Sobolev embedding \eqref{Sob-emb} gives $\|u\|_{L^{p+1}}\lesssim \|u\|_{H^1_{\alpha}}$. Therefore, if we suppose that there exists a sequence of times $\{t_n\}_n\subset(-T_{\min}, T_{\max})$ such that $\|u(t_n)\|_{H^1_{\alpha}}\to 0$ as $n\to\infty$, then also  $|P(u(t_n))
|\lesssim \|u(t_n)\|_{H^1_{\alpha}}^2+\|u(t_n)\|_{H^1_{\alpha}}^{p+1}\to 0$ which is a contradiction with respect to the conclusion of Proposition \ref{prop:P-neg}.
\end{proof}
\begin{proof}[Proof of Theorem \ref{thm-blowup}]
We can now conclude with the proof of Theorem \ref{thm-blowup}. It follows by a convexity argument after \eqref{blow-est}, \eqref{Q-control-ref}, and \eqref{blow-est-2}
\end{proof}

\begin{remark}
We also remark that, as for the classical NLS with free Laplacian, we have global existence of solutions:\vspace{2mm}
    
\noindent   \textup{(i)} provided $p<3$ (mass-subcritical case) and no matter the size of the initial data. Indeed, by the interpolation inequality $\|u\|_{L^{p+1}}\leq C_{\GN} \|u\|_{L^2}^{\frac{2}{p+1}}\|u\|_{H^1_{\alpha}}^{1-\frac{2}{p+1}}$, where $C_{\GN}$ is the Gagliardo-Nirenberg optimal constant, we have 
\begin{equation*}
\begin{aligned}
\|u\|_{H^1_{\alpha}}^2&=\|\nabla \phi^\lambda\|_{L^2}^2+\lambda\| \phi^\lambda\|_{L^2}^2+\Gamma^\lambda_\alpha|q|^2=F(u)+\lambda\|u\|_{L^2}^2\\
&=2E(u)+\lambda M(u)+\frac{2}{p+1}\|u\|_{L^{p+1}}^{p+1} \lesssim 1+\|u\|_{H^1_{\alpha}}^{p-1},
\end{aligned}
\end{equation*}
hence we have a uniform bound on the $H^1_{\alpha}$-norm of the solution;\medskip
    
\noindent    \textup{(ii)} $p=3$ (mass-critical case) provided $\|u_0\|_{L^2}$ is sufficiently small. Indeed, as above we estimate
\begin{equation*}
\begin{aligned}
\|u\|_{H^1_{\alpha}}^2 &= 2E(u) +\lambda M(u)   +\frac{2}{p+1}\|u\|_{L^{p+1}}^{p+1} \\
&\leq 2E(u) +\lambda M(u)  +  \frac{1}{2} C_{\GN}^{4} M(u)\|u\|^2_{H^1_{\alpha}},
\end{aligned}
\end{equation*}
so 
\[
M(u)<2C^{-4}_{\GN}
\] 
guarantees a uniform bound on the $H^1_\alpha$ and the proof is complete.
\end{remark}

\begin{ackno}\rm
V.G. is  partially supported  by the project PRIN  2020XB3EFL by the Italian Ministry of Universities and Research, by   the Top Global University Project, Waseda University, by the University of Pisa, Project PRA 2022 85, and by Institute of Mathematics and Informatics, Bulgarian Academy of Sciences.
\end{ackno}

\bibliographystyle{plain}
\bibliography{Sing}
	
\end{document}